\documentclass{amsart}
\usepackage[usenames,dvipsnames]{color}

\newtheorem{thm}{Theorem}[section]
\newtheorem{cor}[thm]{Corollary}
\newtheorem{lem}[thm]{Lemma}
\newtheorem{prop}[thm]{Proposition}
\theoremstyle{definition}
\newtheorem{defn}[thm]{Definition}
\theoremstyle{remark}
\newtheorem{rem}[thm]{Remark}
\numberwithin{equation}{section}

\newcommand{\paren}[1]{\left(#1\right)}
\newcommand{\corch}[1]{\left[#1\right]}

\newcommand{\N}{\hbox{\ensuremath{\mathbb{N}}}}

\thanks{* Corresponding author.\\
\textit{E-mail addresses:}  sheinek@dm.uba.ar (S. B. Heineken),
jllarena@unsl.edu.ar (J. P. Llarena),\\ morillas@unsl.edu.ar (P. M.
Morillas)}

\begin{document}

\title[On the minimizers of the fusion frame potential]{On the minimizers of the fusion frame potential}

\author[Sigrid B. Heineken]{Sigrid B. Heineken$^{1,*}$}%

\author[Juan P. Llarena]{Juan P. Llarena$^{2}$}%

\author[Patricia M. Morillas]{Patricia M. Morillas$^2$\\\\ $^{1}$\textit{D\lowercase{epartamento de} M\lowercase{atem\'atica}, FCE\lowercase{y}N, U\lowercase{niversidad de }B\lowercase{uenos} A\lowercase{ires}, P\lowercase{abell\'on} I, C\lowercase{iudad }U\lowercase{niversitaria}, IMAS, UBA-CONICET, C1428EGA C.A.B.A. B\lowercase{uenos} A\lowercase{ires}, A\lowercase{rgentina}\\ $^2$ I\lowercase{nstituto de }M\lowercase{atem\'{a}tica }A\lowercase{plicada} S\lowercase{an} L\lowercase{uis, }UNSL-CONICET \lowercase{and} D\lowercase{epartamento de} M\lowercase{atem\'{a}tica}, FCFM\lowercase{y}N, UNSL, E\lowercase{j\'{e}rcito de los }A\lowercase{ndes 950, 5700 }S\lowercase{an} L\lowercase{uis,} A\lowercase{rgentina}}}
%


\begin{abstract}
We study the minimizers of the fusion frame potential in the case
that both the weights and the dimensions of the subspaces are fixed
and not necessarily equal. Using a concept of irregularity we
provide a description of the local (that are also global) minimizers
which projections are eigenoperators of the fusion frame operator.
This result will be related to the existence of tight fusion frames.
In this way we generalize results known for the classical vector
frame potential.

\bigskip

\bigskip

{\bf Key words:} Frames, Fusion frames, Fusion frame potential,
Tight fusion frames, Fundamental inequality, Irregularity.

\medskip

{\bf AMS subject classification:} Primary 42C15; Secondary 42C99,
42C40, 15A60.


\end{abstract}

\maketitle

\section{Introduction}

\textit{Fusion frames} (or \textit{frames of subspaces})
 for a separable Hilbert space $\mathcal{H}$ are collections of closed subspaces and weights,
that allow the reconstruction of each element of $\mathcal{H}$ from
packets of coefficients (see \cite{Casazza-Kutyniok (2004),
Casazza-Kutyniok-Li (2008)} and also \cite[Chapter
13]{Casazza-Kutyniok (2012)}). They are a generalization of
\textit{frames} \cite{Casazza (2000), {Casazza-Kutyniok (2012)},
Christensen (2003), Kovacevic-Chebira (2008)} and are useful in
situations where a distributed processing by combining locally data
vectors has to be implemented, such as distributing sensing,
parallel processing and packet encoding.

In the present paper we consider the finite-dimensional case.
Finite-dimensional Hilbert spaces and finite fusion frames
\cite{Casazza-Kutyniok (2012)} are an important tool in applications
since they prevent the approximation problems that appear with the
truncations in the infinite-dimensional case.





The concept of frame potential has been introduced by Benedetto and
Fickus in \cite{Benedetto-Fickus (2003)}. There the authors give a
complete description of the structure of the unit-norm  minimizers
of this potential, and then characterize the existence of normalized
tight frames in terms of these minimizers. This frame potential was
also considered in \cite{Casazza-Fickus-Kovacevic-Leon-Tremain
(2006)} where the results of \cite{Benedetto-Fickus (2003)} are
extended to the non-uniform length setting.

In \cite{Casazza-Fickus (2009)}, Casazza and Fickus introduced and
studied the notion of \textit{fusion frame potential}. This
potential has been also extensively studied in
\cite{Massey-Ruiz-Stojanoff (2010)}. In this paper we use a
different approach. We generalize results of
\cite{Casazza-Fickus-Kovacevic-Leon-Tremain (2006)} to the setting
of fusion frames. We study the minimizers of the fusion frame
potential in the case that both the weights and the dimensions of
the subspaces are fixed and not necessarily equal.

The organization of the paper is as follows. In section 2 we review
concepts about frames, fusion frames and fusion frame potential. In
Section 3 and 4, we provide a description of the local minimizers
which projections are eigenoperators of the fusion frame operator.
This description is completed in Section 5 using a concept of
irregularity and it is shown that these local minimizers are also
global. In section 6, we relate the study of the minimizers of the
fusion frame potential to the existence of tight fusion frames.


\section{Preliminaries}

In this section we recall the concepts of frame \cite{Casazza
(2000), Casazza-Kutyniok (2012), Christensen (2003),
Kovacevic-Chebira (2008)}, fusion frame \cite{Casazza-Kutyniok
(2004), Casazza-Kutyniok-Li (2008)} (see also \cite[Chapter
13]{Casazza-Kutyniok (2012)}) and fusion frame potential
\cite{Casazza-Fickus (2009), Massey-Ruiz-Stojanoff (2010)}. We refer
to the mentioned works for more details. We begin introducing some
notation.

\subsection{Notation}

Let $\mathbb{F}=\mathbb{R}$ or $\mathbb{F}=\mathbb{C}$. Let $d, n
\in \mathbb{N}$. Noting that any finite dimensional Hilbert space
over $\mathbb{F}$ of dimension $d$ is isomorphic to $\mathbb{F}^{d}$
with the usual inner product we will work directly with this last
Hilbert space.

We will identify the linear transformations from $\mathbb{F}^{d}$ to
$\mathbb{F}^{n}$ with their matrix representation with respect to
the standard bases of $\mathbb{F}^{d}$ and $\mathbb{F}^{n}$.
$\mathbb{F}^{d \times n}$ denotes the set of matrices of order $d
\times n$ with entries in $\mathbb{F}$. If $M \in \mathbb{F}^{d
\times n}$, then $R(M)$ and $M^{*}$ denotes the range and the
conjugate transpose of $M$, respectively. The elements of
$\mathbb{F}^{n}$ will be considered as column vectors, i.e., we
identify $\mathbb{F}^{n}$ with $\mathbb{F}^{n \times 1}$, and if $f
\in \mathbb{F}^{n}$ then $f(i)$ denotes the $i$th component of $f$.
Let $M \in \mathbb{F}^{d \times n}$. We denote the entry $i,j$ of
$M$ with $M(i,j)$ and the  $j$th column by $M(:,j).$ The inner
product and the norm in $\mathbb{F}^{d}$ will be denoted by
$\langle.,.\rangle$ and $\|.\|$, respectively. If $T \in
\mathbb{F}^{d\times n}$, then $\|T\|_{F}$ denotes the Frobenius norm
of $T$.

The set of subspaces of $\mathbb{F}^{d}$ will be denoted by
$\mathcal{S}$. If $V \in \mathcal{S}$, $\pi_{V} \in \mathbb{F}^{d
\times d}$ denotes the orthogonal projection onto $V$.

\subsection{Frames and fusion frames}

\begin{defn}\label{D F}
Let $\mathcal{F}=\{f_{k}\}_{k=1}^{K} \subset \mathbb{F}^{d}$.
\begin{enumerate}
  \item $\mathcal{F}$ is a \textit{frame} for $\mathbb{F}^{d}$ if
$\text{span}~\mathcal{F}=\mathbb{F}^{d}$.
  \item If $\mathcal{F}$ is a frame for $\mathbb{F}^{d},$

\centerline{ $S_{\mathcal{F}} : \mathbb{F}^{d} \rightarrow
\mathbb{F}^{d},$ $S_{\mathcal{F}}f=\sum_{k=1}^{K}\langle
f,f_{k}\rangle f_{k},$} \noindent is the \textit{frame operator} of
$\mathcal{F}$.
  \item $\mathcal{F}$ is an $\alpha$-\textit{tight frame}, if
$S_{\mathcal{F}}=\alpha I$. If $S_{\mathcal{F}}=I$, $\mathcal{F}$ is
a \textit{Parseval frame}.
\end{enumerate}
\end{defn}

If $\mathcal{F}$ is an $\alpha$-tight frame we have the following
reconstruction formula
\begin{equation}\label{E F tight formula de reconstruccion}
f=\frac{1}{\alpha}\sum_{k=1}^{K}\langle f,f_{k}\rangle f_{k}, \text{
for all $f \in \mathbb{F}^{d}$,}\end{equation} where the scalar
coefficients $\langle f,f_{k}\rangle$, $k = 1, \ldots, K$, can be
thought as a measure of the projection of $f$ onto each frame
vector.

Fusion frames generalize the notion of frames. Each $f \in
\mathbb{F}^{d}$ can be a represented via fusion frames is given by
projections onto multidimensional subspaces. Before defining them we
introduce the following concept.
\begin{defn}
Let $\{W_k\}_{k=1}^{K}$ be a family of subspaces of
$\mathbb{F}^{d}$, and let $\{w_{k}\}_{k=1}^{K}$ be a family of
weights, i.e., $w_{k}
> 0$ for $k=1,\ldots,K$. Then $\{(W_k,w_{k})\}_{k=1}^{K}$ is
called a \textit{Bessel fusion sequence} for $\mathbb{F}^{d}$.
\end{defn}

Throughout the article we will assume that $K$ is fixed and we
denote $\{W_{k}\}_{k=1}^{K}\subset\mathcal{S}$ with $\mathbf{W}$,
$\{w_{k}\}_{k=1}^{K}$ with $\mathbf{w}$ and
$\{(W_k,w_{k})\}_{k=1}^{K}$ with $(\mathbf{W},\mathbf{w})$. If
$w_{1} = \dots = w_{K}=w$ we write $w$ instead of $\mathbf{w}$. The
set of Bessel fusion sequences in $\mathbb{F}^{d}$ with $K$
subspaces of $\mathbb{F}^{d}$ and weights will be denoted with
$\mathcal{B}_{K}$.

Let $\mathcal{W}:=\bigoplus_{k=1}^{K} W_k = \{(f_{k})_{k=1}^{K}:
f_{k} \in W_k\}$ be the Hilbert space with
$$\langle(f_{k})_{k=1}^{K},(g_{k})_{k=1}^{K}\rangle=\sum_{k=1}^{K}\langle f_{k}, g_{k}\rangle.$$
\begin{defn}\label{D FF}
Let $(\mathbf{W},\mathbf{w}) \in \mathcal{B}_{K}$.
\begin{enumerate}
  \item $(\mathbf{W},\mathbf{w})$ is called a
  \textit{fusion frame} for $\mathbb{F}^{d}$ if
$\text{span}\bigcup_{k=1}^{K}W_k=\mathbb{F}^{d}$.
  \item $(\mathbf{W},\mathbf{w})$ is a \textit{Riesz
fusion basis} if $\mathbb{F}^{d}$ is the direct sum of the $W_k$,
and $(\mathbf{W},1)$ an \textit{orthonormal fusion basis} if
$\mathbb{F}^{d}$ is the orthogonal sum of the $W_k.$
  \item If $(\mathbf{W},\mathbf{w})$  is a fusion frame for $\mathbb{F}^{d}$, the
operator

\centerline{ $S_{\mathbf{W},\mathbf{w}}: \mathbb{F}^{d} \rightarrow
\mathbb{F}^{d}\text{ ,
}S_{\mathbf{W},\mathbf{w}}(f)=\sum_{k=1}^{K}w_{k}^{2}\pi_{W_k}(f)$}

\noindent is called the \textit{fusion frame operator} of
$(\mathbf{W},\mathbf{w})$.

  \item A fusion frame $(\mathbf{W},\mathbf{w})$ is called an
$\alpha$-\textit{tight fusion frame} if
$S_{\mathbf{W},\mathbf{w}}=\alpha I$. If
$S_{\mathbf{W},\mathbf{w}}=I$ we say that it is a \textit{Parseval
fusion frame}.
\end{enumerate}
\end{defn}

If $(\mathbf{W},\mathbf{w})$ is an $\alpha$-tight fusion frame we
have the following reconstruction formula
\begin{equation}\label{E FF tight formula de reconstruccion}
f=\frac{1}{\alpha}\sum_{k=1}^{K}w_{k}^{2}\pi_{W_k}(f), \text{ for
all $f \in \mathbb{F}^{d}$.}
\end{equation}

\subsection{The fusion frame potential}

The set of sequences of $K$ subspaces of $\mathbb{F}^{d}$ will be
denoted with $\mathcal{S}_{K}$. In what follows we assume that the
weights $\mathbf{w}$ are fixed. Let $K \in \mathbb{N}$ and
$\mathbf{L}=(L_{1},\dots,L_{K}) \in \mathbb{N}^{K}$. If $L_{1} =
\dots = L_{K}=L$ we write $L$ instead of $\mathbf{L}$. We consider
the set

\centerline{ $\mathcal{S}_{K}(\mathbf{L}):=\{\mathbf{W} \in
\mathcal{S}_K: \text{dim}(W_{k}) = L_k \text{ for every } k \in \{1,
\ldots, K\}\}$.}

We define the distance between $W, V \in \mathcal{S}$ as
$$d(W, V)=\|\pi_{W}-\pi_{V}\|_{F}$$ and the distance between $\mathbf{W}, \mathbf{V} \in
\mathcal{S}_{K}$ as
$$d(\mathbf{W}, \mathbf{V})=\left[\sum_{k=1}^{K}\|\pi_{W_{k}}-\pi_{V_{k}}\|_{F}^{2}\right]^{1/2}.$$

With this metric the set $\mathcal{S}_{K}(\mathbf{L})$ is compact
(this follows from \cite[Lemma 4.3.1]{Massey-Ruiz-Stojanoff
(2010)}).

Following \cite{Casazza-Fickus (2009)} we define the fusion frame
potential
$FFP_{\mathbf{w}}:\mathcal{S}_{K}(\mathbf{L})\rightarrow\mathbb{R}$
as
\begin{equation}\label{E FFP}
FFP_{\mathbf{w}}(\mathbf{W})=\text{tr}(S_{\mathbf{W},\mathbf{w}}^{2}).
\end{equation}

Note that

\centerline{$\text{tr}(S_{\mathbf{W},\mathbf{w}}^{2})=\text{tr}(\displaystyle\sum_{k=1}^Kw_k^2\pi_{W_{k}}\displaystyle\sum_{k=1}^Kw_k^2\pi_{W_{k}})=\displaystyle\sum_{k,k^{\prime}=1}^Kw_k^2w_{k\prime}^2\text{tr}(\pi_{W_{k}}\pi_{W_{k^{\prime}}})$}

\noindent and each term in the above sum is a nonnegative real
number since

\begin{align*}
\text{tr}(\pi_{W_{k}}\pi_{W_{k^{\prime}}})&=\text{tr}(\pi_{W_{k}}^{2}\pi_{W_{k^{\prime}}}^{2})=\text{tr}(\pi_{W_{k}}\pi_{W_{k^{\prime}}}\pi_{W_{k^{\prime}}}\pi_{W_{k}})=\text{tr}(\pi_{W_{k}}^{*}\pi_{W_{k^{\prime}}}^{*}\pi_{W_{k^{\prime}}}\pi_{W_{k}})\\&=\text{tr}((\pi_{W_{k^{\prime}}}\pi_{W_{k}})^{*}\pi_{W_{k^{\prime}}}\pi_{W_{k}})=\|\pi_{W_{k^{\prime}}}\pi_{W_{k}}\|_{F}^{2}.
\end{align*}

Analogous as in \cite[Proposition 1]{Casazza-Fickus (2009)} the
following can be proved:

\begin{prop}\label{cotainf}
If $\mathbf{W} \in \mathcal{S}_{K}(\mathbf{L})$, then
\begin{equation}\label{ci}
FFP_{\mathbf{w}}(\mathbf{W})\geq
\frac{1}{d}\left(\sum_{k=1}^{K}w_k^2L_k\right)^2,
\end{equation}
with equality holding in (\ref{ci}) if and only if
$(\mathbf{W},\mathbf{w})$ is a tight fusion frame for
$\mathbb{F}^{d}.$
\end{prop}

\section{The orthogonal projections as eigenoperators
of the fusion frame operator}

From now on we denote the different eigenvalues of
$S_{\mathbf{W},\mathbf{w}}$ listed in decreasing order with
$\{\lambda_j\}_{j=1}^J$ and the corresponding sequence of
eigenspaces with $\{E_j\}_{j=1}^J$. We consider the index sets
$$I_j=\{k\in\{1,\ldots,K\}:S_{\mathbf{W},\mathbf{w}}\pi_{W_{k}}=\lambda_j\pi_{W_{k}}\}.$$

We define the class of Bessel fusion sequences for which the
orthogonal projections are ``eigenoperators" of the fusion frame
operator:
$$\mathcal{E}=\{(\mathbf{W},\mathbf{w})\in \mathcal{B}_{K}: \forall k\in\{1,\ldots,K\}\,\, \exists j \in \{1, \ldots,
J\} \text{ such that }
S_{\mathbf{W},\mathbf{w}}\pi_{W_{k}}=\lambda_j\pi_{W_{k}}\}.$$ For
$L_1=\cdots=L_K=1,$  it can be concluded from the vectorial case
that every local minimizer of
$FFP_{\mathbf{w}}:\mathcal{S}_{K}(1)\rightarrow\mathbb{R}$ belongs
to $\mathcal{E}$ (see \cite{Casazza-Fickus-Kovacevic-Leon-Tremain
(2006)}). It is natural to ask if this is true  in the fusion frame
setting for arbitrary $\mathbf{L}.$ The answer is no (we will deal
with this question in Remark~\ref{contraejemplo}), but one can
assert something weaker. Let $\mathbf{W}$ be a local minimizer of
$FFP_{\mathbf{w}}$. Theorem~4 in \cite{Casazza-Fickus (2009)} says
that there exists a generating family
$\{\{f_{k,l}\}_{l=1}^{L_k}\}_{k=1}^{K}$ of
$\{\pi_{W_{k}}\}_{k=1}^{K}$ (i.e. for each  $k = 1, \ldots, K$,
$\{f_{k,l}\}_{l=1}^{L_k}$ is an orthonormal basis for $W_k$) that
consists of eigenvectors of $S_{\mathbf{W},\mathbf{w}},$ and that
the elements of this family which are in $E_{j}$ are a
$\lambda_{j}$-tight frame for $E_{j}$.

The next theorem establishes properties of the index sets $I_j$
along with properties of the families $\{W_k\}_{k \in I_j},$
$j=1,\ldots,J$ for pairs $(\mathbf{W},\mathbf{w})\in \mathcal{E}$:

\begin{thm}\label{T particion minimizadores FFP en TFF}
Let $(\mathbf{W},\mathbf{w})\in \mathcal{E}.$

The following assertions hold:
\begin{enumerate}
  \item $\bigcup_{j=1}^{J}I_j=\{1,\ldots,K\}$.
  \item $k \in I_j$ if and only if $W_k \subseteq
E_j.$
  \item If $j\not=j^{\prime}$, then $I_j\cap I_{j^{\prime}}=\emptyset$.
  \item $\lambda_{J}=\frac{1}{\text{dim}(E_{J})}\sum_{k \in
  I_{J}}w_{k}^{2}L_{k}$.
  \item For all
$j=1,\ldots,J$, $\{W_k\}_{k \in I_j}$ is a $\lambda_j$-tight fusion
frame for $E_j$.
  \item $I_j = \emptyset$ if and only if $\lambda_{j} = 0$.
\end{enumerate}
\end{thm}

\begin{proof}
(1) It follows from the fact that $(\mathbf{W},\mathbf{w})\in
\mathcal{E}$ and the definition of $I_{j}$.

(2) If $k \in I_j$ and $f \in W_k$, then
$S_{\mathbf{W},\mathbf{w}}f=S_{\mathbf{W},\mathbf{w}}\pi_{W_{k}}f=\lambda_j\pi_{W_{k}}f=\lambda_jf$.
Thus $f \in E_j$. Conversely, suppose that $W_k \subseteq E_j$. If
$f \in \mathcal{H}$,
$S_{\mathbf{W},\mathbf{w}}\pi_{W_{k}}f=\lambda_j\pi_{W_{k}}f$. Thus
$k \in I_j$.

(3) It follows from (2) and the orthogonality of the eigenspaces
$E_j$.

(4) Let $\{e_{l}\}_{l=1}^{\text{dim}(E_{J})}$ be an orthonormal
basis for $E_{J}$ and $\{e_{l}\}_{l=\text{dim}(E_{J})+1}^{d}$ be an
orthonormal basis for $E_{J}^{\bot}$. By part (2), $P_{k}e_{l}=0$
for $k \in I_{J}$ and $l\in \{\text{dim}(E_{J})+1, \ldots, d\}$, and
$P_{k}e_{l}=0$ for $k \notin I_{J}$ and $l \in \{1, \ldots,
\text{dim}(E_{J})\}$, hence
\begin{align*}
\sum_{k \in I_{J}}w_{k}^{2}L_{k} &= \sum_{k \in
I_{J}}w_{k}^{2}\text{tr}(P_{k}) = \sum_{k \in I_{J}}w_{k}^{2}\sum_{l=1}^{d}\langle P_{k}e_{l},e_{l} \rangle =  \sum_{k \in I_{J}}w_{k}^{2}\sum_{l=1}^{\text{dim}(E_{J})}\langle P_{k}e_{l},e_{l} \rangle\\
   &=  \sum_{l=1}^{\text{dim}(E_{J})}\langle \sum_{k=1}^{K}w_{k}^{2}P_{k}e_{l},e_{l}
   \rangle = \sum_{l=1}^{\text{dim}(E_{J})}\langle S_{\mathbf{W},\mathbf{w}}e_{l},e_{l}
   \rangle\\
   &= \lambda_{J}\sum_{l=1}^{\text{dim}(E_{J})}\langle I_{\mathcal{H}}e_{l},e_{l}\rangle = \lambda_{J}\sum_{l=1}^{\text{dim}(E_{J})}\| e_{l}\|^{2} =
\lambda_{J} \text{dim}(E_{J}).
\end{align*}

(5) Let $j\in\{1,\ldots,J\}$ and $f\in E_j$. By (1) and (2), if $k
\in \{1, \ldots, K\} \cap I_{j}^{c}$ then $W_k\subseteq
E_{j^{\prime}}$ for some $j^{\prime}\not=j$ and consequently
$\pi_{W_{k}}f=0$. So,
$$\lambda_j f = S_{\mathbf{W},\mathbf{w}}f= \sum_{k=1}^Kw_k^2\pi_{W_{k}}f
= \sum_{k \in I_j}w_k^2\pi_{W_{k}}f.$$ This shows that $\{W_{k}\}_{k
\in I_j}$ is a $\lambda_j$-tight fusion frame for $E_j$.

(6) We have $I_{j_{0}} = \emptyset$ if and only of $\forall k \in
\{1, \ldots, K\}, W_{k} \subseteq E_{j_{0}}^{\bot}$, or
equivalently, $E_{j_{0}} \subseteq
R(T_{\mathbf{W},\mathbf{w}})^{\bot} = N(S_{\mathbf{W},\mathbf{w}})$
i. e. $\lambda_{j_{0}} = 0 $
\end{proof}

\section{Minimizers in $\mathcal{E}$}

In the vectorial case every sequence, which is a local minimizer of
the frame potential, consists of elements that are eigenvectors of
the frame operator. Motivated by that, in this section we describe
the structure of the local minimizers $\mathbf{W}$ of the fusion
frame potential with fixed dimensions and weights such that
$(\mathbf{W}, \mathbf{w})\in \mathcal{E}.$ In this case we can give
a complete description as it was done in \cite{Benedetto-Fickus
(2003), Casazza-Fickus-Kovacevic-Leon-Tremain (2006)} for the frame
potential. Moreover, as in this works, we relate these results to
the existence of tight fusion frames in terms of a fundamental
inequality that must be satisfied by the weights and the dimension
of the subspaces.

To prove Theorem~\ref{T particion minimizadores FFP en sumas
directas} we use the following auxiliary Lemma.

\begin{lem}\label{L D R(A)}
Let $\{f_{l}\}_{l=1}^{L} \subset \mathbb{F}^{d}$ be an orthonormal
set. Let $\{z_{l}\}_{l=1}^{L} \subset \mathbb{C}$ be a nonzero
sequence such that $|z_{l}|\leq \frac{1}{2}$ for each $l = 1,
\ldots, L$. Let $h \in \mathbb{F}^{d}$ such that $\|h\|=1$ and
$\langle h, f_{l}\rangle = 0$ for each $l=1, \ldots, L$. For each
$t\in (-1,1)$, let the matrices $A(t)$, $F$, $H$, $\widetilde{F}$
and $D$ given by

\centerline{$A(t)(:,l)=(1-t^2|z_{l}|^2)^{\frac{1}{2}}f_{l}+tz_{l}h$
,}

\centerline{$F(:,l)=f_{l}$,}

\centerline{$H(:,l)=z_{l}h$,}

\centerline{$\widetilde{F}(:,l)=-|z_{l}|^{2}f_{l},$}

\noindent for $l=1, \ldots, L,$ and

\centerline{$D(l,l')=-2\delta_{l,l'}z_{l}\overline{z_{l'}}$ for $l,
l' = 1, \ldots, L$.}

\noindent Then
\begin{enumerate}
  \item $\text{dim}(\text{R}(A(t)))=L$.
  \item $\frac{d}{dt}\pi_{\text{R}(A(t))}|_{t=0}=HF^{*}+FH^{*}$.
  \item
  $\frac{d^{2}}{dt^{2}}\pi_{\text{R}(A(t))}|_{t=0}=\widetilde{F}F^{*}+2HH^{*}+FDF^*+F\widetilde{F}^{*}$.
\end{enumerate}
\end{lem}

\begin{proof}
To prove (1) we are going to see that the columns of $A(t)$ are
linearly independent. To see this, let $\{c_{l}\}_{l=1}^{L} \subset
\mathbb{C}$ be such that
$\displaystyle\sum_{l=1}^{L}c_lA(t)(:,l)=0$, or equivalently,
$\displaystyle\sum_{l=1}^{L}c_l(1-t^2|z_{l}|^2)^{\frac{1}{2}}f_{l}+t\displaystyle\sum_{l=1}^{L}c_lz_{l}h=0$.
Let $l_{0} \in \{1, \ldots, L\}$. Taking into account that $\langle
f_{l_0},f_{l}\rangle=\delta_{l,l_0}$ for all $l \in \{1, \ldots,
L\}$ and $\langle f_{l_0},h \rangle=0$, then $0=\langle
f_{l_0},\displaystyle\sum_{l=1}^{L}c_l(1-t^2|z_{l}|^2)^{\frac{1}{2}}f_{l}+t\displaystyle\sum_{l=1}^{L}c_lz_{l}h\rangle=c_{l_0}(1-t^2|z_{l}|^2)^{\frac{1}{2}}$.
Since $t\in (-1,1)$ and $|z_{l}|\leq\frac{1}{2}$ for each $k \in
\{1, \ldots, K\}$ and $l=1,\ldots, L$, it follows that $c_{l_0}=0$.
Now we are going to prove (2) and (3). We begin noting that
\begin{equation}\label{E derivada Ak en 0}
\frac{d}{dt}A(t)|_{t=0}=H
\end{equation}
\begin{equation}\label{E derivada segunda Ak en 0}
\frac{d^2}{dt^2}A(t)|_{t=0}=\widetilde{F}
\end{equation}
and

\centerline{$[A^*(t)A(t)](l,l')=\delta_{l,l'}+(1-\delta_{l,l'})t^2z_{l}\overline{z_{l'}}
\text{ for $l,l' = 1, \ldots, L$}.$}

\noindent Let $\mathbb{L}:=\{1, \ldots,L\}$, $B \subset \mathbb{L}$,
$\mathbb{L}^{s}\setminus B=\{(l_1,\ldots,l_s) \in \mathbb{L}
\setminus B \times \ldots \times \mathbb{L} \setminus B : l_1 <
\ldots < l_s\}$. We write $\mathbb{L}^{s}\setminus
\emptyset:=\mathbb{L}^{s}$. It is easy to check that the entry
$(i,i)$ of $[A^*(t)A(t)]^{-1}$ is

\centerline{$\frac{1+\displaystyle\sum_{s=2}^{L-1}(-1)^{s-1}(s-1)t^{2s}\sum_{(l_1,\ldots,l_s)\in
\mathbb{L}^{s}\setminus\{i\}}|z_{l_1}|^{2}\ldots|z_{l_s}|^{2}}{1+\displaystyle\sum_{s=2}^{L}(-1)^{s-1}(s-1)t^{2s}\sum_{(l_1,\ldots,l_s)\in
\mathbb{L}^{s}}|z_{l_1}|^{2}\ldots|z_{l_s}|^{2}}$}

\noindent and the entry $(i,j)$ with $i \neq j$ is

\centerline{$\frac{-t^{2}\bar{z}_{i}z_{j}\left(1+\displaystyle\sum_{s=1}^{L-2}(-1)^{s}t^{2s}\sum_{{(l_1,\ldots,l_s)\in
\mathbb{L}^{s}}\setminus\{i,j\}}|z_{l_1}|^{2}\ldots|z_{l_s}|^{2}\right)}{1+\displaystyle\sum_{s=2}^{L}(-1)^{s-1}(s-1)t^{2s}\sum_{(l_1,\ldots,l_s)\in
\mathbb{L}^{s}}|z_{l_1}|^{2}\ldots|z_{l_s}|^{2}}.$}

\noindent Therefore,
\begin{equation}\label{E derivada inversa Ak* Ak en 0}
\frac{d}{dt}([A^{\ast}(t)A(t)]^{-1})|_{t=0}=0
\end{equation}
and
\begin{equation}\label{E derivada segunda inversa Ak* Ak en 0}
\frac{d^{2}}{dt^{2}}([A^{\ast}(t)A(t)]^{-1})|_{t=0}=D.
\end{equation}
We have
$\pi_{\text{R}(A(t))}=A(t)[A^{\ast}(t)A(t)]^{-1}A^{\ast}(t)$. Thus,
from (\ref{E derivada Ak en 0}), (\ref{E derivada segunda Ak en 0}),
(\ref{E derivada inversa Ak* Ak en 0}) and (\ref{E derivada segunda
inversa Ak* Ak en 0}) we obtain (2) and (3).
\end{proof}

\begin{thm}\label{T particion minimizadores FFP en sumas directas}
If $\mathbf{W}$ is a local minimizer of $FFP_{\mathbf{w}}$  such
that $(\mathbf{W},\mathbf{w})\in \mathcal{E}$, then $\{W_k\}_{k \in
I_j}$ is a direct sum for all $j<J$.
\end{thm}

\begin{proof}
Suppose that there exists $1 \leq j \leq J-1$ such that $\{W_k\}_{k
\in I_j}$ is not a direct sum. Then there exists a linearly
dependent set $\{\{f_{k,l}\}_{l=1}^{L_{k}}\}_{k \in I_j}$ such that
$\{f_{k,l}\}_{l=1}^{L_k}$ is an orthonormal basis for $W_k$, $k \in
I_j$. Consequently, there exists a nonzero sequence
$\{\{z_{k,l}\}_{k=1}^{L_{k}}\}_{k \in I_j} \subset \mathbb{C}$ with
$|z_{k,l}|\leq \frac{1}{2}$ for all $k \in I_j$, $k=1,\ldots,L_k$
such that
\begin{equation}\label{E sum zkl fkl 0}
\displaystyle\sum_{k \in
I_j}w_{k}^{2}\sum_{l=1}^{L_k}\overline{z}_{k,l}f_{k,l}=0.
\end{equation}
The eigenspace $E_J$ of $S_{\mathbf{W},\mathbf{w}}$ is at least
1-dimensional. So, there exists $h\in E_J$ with $\| h\|=1$. Using
$\{f_{k,l}\}_{l=1}^{L_{k}}$, $\{z_{k,l}\}_{k=1}^{L_{k}}$ and $h$, we
define $A_k(t)$, $F_k$, $H_k$, $\widetilde{F}_k$ and $D_k$ as in
Lemma~\ref{L D R(A)} where we now consider $k \in I_j$ instead of $k
\in \{1, \ldots, K\}$.

Since the columns of $F_{k}$ are orthonormal,
\begin{equation}\label{E Fk* Fk = I}
F_k^{*}F_{k}=I  \text{ for $k \in I_{j}$.}
\end{equation}
By Theorem~\ref{T particion minimizadores FFP en TFF}(2), if $k \in
I_{j}$ then $\{\{f_{k,l}\}_{l=1}^{L_{k}}\}_{k \in I_j} \subset
E_{j}$. Thus, the conclusions of Lemma~\ref{L D R(A)} hold and
\begin{equation}\label{E Hk* Fk' = 0}
H_{k}^{*}F_{k'}=0 \text{ for $k, k' \in I_{j}$.}
\end{equation}
Let
$$V_k(t)=\left\{
    \begin{array}{ll}
      \text{R}(A_k(t)), & \hbox{si $k \in I_j$;} \\
      W_k, & \hbox{si $k \not\in I_j$.}
    \end{array}
  \right.$$
The sequence of continuous subspaces curves $\{V_k(t)\}_{k=1}^K$
satisfies that $\{V_k(0)\}_{k=1}^K=\mathbf{W}$ and, by Lemma~\ref{L
D R(A)}(1), $\{V_k(t)\}_{k=1}^K \in \mathcal{S}_{K}(\mathbf{L})$ for
all $t\in (-1,1)$. We are going to see that
$$\frac{d}{dt}FFP_{\mathbf{w}}(\{V_k(t)\}_{k=1}^K)|_{t=0}=0 \text{
and }
\frac{d^2}{dt^2}FFP_{\mathbf{w}}(\{V_k(t)\}_{k=1}^K)|_{t=0}<0.$$
Therefore,
\begin{equation} FFP_{\mathbf{w}}(\{V_k(t)\}_{k=1}^K)< FFP_{\mathbf{w}}(\mathbf{W})
\end{equation}
for all $t$ in a neighborhood of $0$, contradicting the hypothesis
that $\mathbf{W}$ is a local minimum of $FFP_{\mathbf{w}}$.

To simplify the exposition in the rest of the proof we set
$P_k:=\pi_{W_k}=F_{k}F_{k}^{*}$ and $\tilde{P}_k(t):=\pi_{V_k(t)}$.

We have
\begin{align*}
\frac{d}{dt}&FFP_{\mathbf{w}}(\{V_k(t)\}_{k=1}^K) = \frac{d}{dt}\sum_{k=1}^{K}\sum_{k'=1}^{K}w_{k}^{2}w_{k'}^{2}\text{tr}([\widetilde{P}_{k}(t)\widetilde{P}_{k'}(t)][\widetilde{P}_{k'}(t)\widetilde{P}_{k}(t)])= \\
=&\sum_{k=1}^{K}\sum_{k'=1}^{K}w_{k}^{2}w_{k'}^{2}\text{tr}\paren{\widetilde{P}_{k}^{'}(t)\widetilde{P}_{k'}(t)\widetilde{P}_{k'}(t)\widetilde{P}_{k}(t)}+\sum_{k=1}^{K}\sum_{k'=1}^{K}w_{k}^{2}w_{k'}^{2}\text{tr}\paren{\widetilde{P}_{k}(t)\widetilde{P}_{k'}^{'}(t)\widetilde{P}_{k'}(t)\widetilde{P}_{k}(t)} \\
&+\sum_{k=1}^{K}\sum_{k'=1}^{K}w_{k}^{2}w_{k'}^{2}\text{tr}\paren{\widetilde{P}_{k}(t)\widetilde{P}_{k'}(t)\widetilde{P}_{k'}^{'}(t)\widetilde{P}_{k}(t)}+\sum_{k=1}^{K}\sum_{k'=1}^{K}w_{k}^{2}w_{k'}^{2}\text{tr}\paren{\widetilde{P}_{k}(t)\widetilde{P}_{k'}(t)\widetilde{P}_{k'}(t)\widetilde{P}_{k}^{'}(t)}\\
=&2(\sum_{k=1}^{K}\sum_{k'=1}^{K}w_{k}^{2}w_{k'}^{2}\text{tr}\paren{\widetilde{P}_{k}^{'}(t)\widetilde{P}_{k'}(t)\widetilde{P}_{k}(t)}+\sum_{k=1}^{K}\sum_{k'=1}^{K}w_{k}^{2}w_{k'}^{2}\text{tr}\paren{\widetilde{P}_{k}^{'}(t)\widetilde{P}_{k}(t)\widetilde{P}_{k'}(t)})
\end{align*}
Then
\begin{eqnarray*}
  \frac{d}{dt}FFP_{\mathbf{w}}(\{V_k(t)\}_{k=1}^K)|_{t=0} &=& 2(\sum_{k \in
I_j}\sum_{k'=1}^{K}w_{k}^{2}w_{k'}^{2}\text{tr}\paren{\widetilde{P}_{k}^{'}(0)P_{k'}P_{k}}+\sum_{k
\in I_j}\sum_{k'=1}^{K}w_{k}^{2}w_{k'}^{2}\text{tr}\paren{\widetilde{P}_{k}^{'}(0)P_{k}P_{k'}})\\
   &=& 2(\sum_{k \in
I_j}w_{k}^{2}\text{tr}\paren{\widetilde{P}_{k}^{'}(0)S_{\mathbf{W},\mathbf{w}}P_{k}}+\sum_{k
\in
I_j}w_{k}^{2}\text{tr}\paren{\widetilde{P}_{k}^{'}(0)P_{k}S_{\mathbf{W},\mathbf{w}}})
\end{eqnarray*}
Using that $(\mathbf{W},\mathbf{w})\in \mathcal{E}$, Lemma~\ref{L D
R(A)}(2), (\ref{E Fk* Fk = I}) and (\ref{E Hk* Fk' = 0}),
\begin{align*}
  \frac{d}{dt}(FFP_{\mathbf{w}}(\{\tilde{P}_k\}_{k=1}^K))|_{t=0} &= 2(\lambda_{j}\sum_{k \in
I_j}w_{k}^{2}\text{tr}(\widetilde{P}_{k}^{'}(0)P_{k})+\lambda_{j}\sum_{k
\in I_j}w_{k}^{2}\text{tr}(\widetilde{P}_{k}^{'}(0)P_{k}))\\
   &=4\lambda_{j}\sum_{k \in
I_j}w_{k}^{2}\text{tr}\paren{(H_{k}F_{k}^{*}+F_{k}H_{k}^{*})F_{k}F_{k}^*}\\
&= 4\lambda_{j}\sum_{k \in
I_j}w_{k}^{2}\text{tr}\paren{H_{k}F_{k}^*}=4\lambda_{j}\sum_{k \in
I_j}w_{k}^{2}\text{tr}\paren{F_{k}^*H_{k}} =0.
\end{align*}
For the second derivative we have
\begin{eqnarray*}\label{E derivada segunda 1}
  \frac{d^2}{dt^2}FFP_{\mathbf{w}}(\{V_k(t)\}_{k=1}^K) &=& \sum_{k=1}^{K}\sum_{k'=1}^{K}w_{k}^{2}w_{k'}^{2}\frac{d^{2}}{dt^{2}}\text{tr}\paren{\corch{\widetilde{P}_{k}(t)\widetilde{P}_{k'}(t)}\corch{\widetilde{P}_{k'}(t)\widetilde{P}_{k}(t)}} \\
   &=& 2[a(t)+b(t)]
\end{eqnarray*}
where
\begin{eqnarray*}
  a(t) &:=& \sum_{k=1}^{K}\sum_{k'=1}^{K}w_{k}^{2}w_{k'}^{2}\text{tr}\paren{\widetilde{P}_{k}^{''}(t)\widetilde{P}_{k'}(t)\widetilde{P}_{k}(t)}+\sum_{k=1}^{K}\sum_{k'=1}^{K}w_{k}^{2}w_{k'}^{2}\text{tr}\paren{\widetilde{P}_{k}^{'}(t)\widetilde{P}_{k'}^{'}(t)\widetilde{P}_{k'}(t)\widetilde{P}_{k}(t)}+\\
  && +\sum_{k=1}^{K}\sum_{k'=1}^{K}w_{k}^{2}w_{k'}^{2}\text{tr}\paren{\widetilde{P}_{k}^{'}(t)\widetilde{P}_{k'}(t)\widetilde{P}_{k'}^{'}(t)\widetilde{P}_{k}(t)}+\sum_{k=1}^{K}\sum_{k'=1}^{K}w_{k}^{2}w_{k'}^{2}\text{tr}\paren{\widetilde{P}_{k}^{'}(t)\widetilde{P}_{k'}(t)\widetilde{P}_{k}^{'}(t)}
\end{eqnarray*}
and
\begin{eqnarray*}
  b(t) &:=& \sum_{k=1}^{K}\sum_{k'=1}^{K}w_{k}^{2}w_{k'}^{2}\text{tr}\paren{\widetilde{P}_{k}^{'}(t)\widetilde{P}_{k'}^{'}(t)\widetilde{P}_{k'}(t)\widetilde{P}_{k}(t)}+\sum_{k=1}^{K}\sum_{k'=1}^{K}w_{k}^{2}w_{k'}^{2}\text{tr}\paren{\widetilde{P}_{k}(t)\widetilde{P}_{k'}^{''}(t)\widetilde{P}_{k'}(t)}+\\
  &&
  +\sum_{k=1}^{K}\sum_{k'=1}^{K}w_{k}^{2}w_{k'}^{2}\text{tr}\paren{\widetilde{P}_{k}(t)\widetilde{P}_{k'}^{'}(t)\widetilde{P}_{k'}^{'}(t)}+\sum_{k=1}^{K}\sum_{k'=1}^{K}w_{k}^{2}w_{k'}^{2}\text{tr}\paren{\widetilde{P}_{k}(t)\widetilde{P}_{k'}^{'}(t)\widetilde{P}_{k'}(t)\widetilde{P}_{k}^{'}(t)}.
\end{eqnarray*}
Using that $(\mathbf{W},\mathbf{w})\in \mathcal{E}$, we obtain
\begin{align}\label{E a(0)}
a(0) =& \sum_{k \in I_j}\sum_{k'=1}^{K}w_{k}^{2}w_{k'}^{2}\text{tr}\paren{\widetilde{P}_{k}^{''}(0)\widetilde{P}_{k'}\widetilde{P}_{k}}+\sum_{k \in I_j}\sum_{k' \in I_j}w_{k}^{2}w_{k'}^{2}\text{tr}\paren{\widetilde{P}_{k}^{'}(0)\widetilde{P}_{k'}^{'}(0)\widetilde{P}_{k'}\widetilde{P}_{k}}+\notag\\
& +\sum_{k \in I_j}\sum_{k' \in
I_j}w_{k}^{2}w_{k'}^{2}\text{tr}\paren{\widetilde{P}_{k}^{'}(0)\widetilde{P}_{k'}\widetilde{P}_{k'}^{'}(0)\widetilde{P}_{k}}+\sum_{k
\in I_j}\sum_{k'=1}^{K}w_{k}^{2}w_{k'}^{2}\text{tr}\paren{\widetilde{P}_{k}^{'}(0)\widetilde{P}_{k'}\widetilde{P}_{k}^{'}(0)}\notag\\
=& \sum_{k \in I_j}w_{k}^{2}\text{tr}\paren{\widetilde{P}_{k}^{''}(0)S_{\mathbf{W},\mathbf{w}}\widetilde{P}_{k}}+\sum_{k \in I_j}\sum_{k' \in I_j}w_{k}^{2}w_{k'}^{2}\text{tr}\paren{\widetilde{P}_{k}^{'}(0)\widetilde{P}_{k'}^{'}(0)\widetilde{P}_{k'}\widetilde{P}_{k}}+\notag\\
& +\sum_{k \in I_j}\sum_{k' \in
I_j}w_{k}^{2}w_{k'}^{2}\text{tr}\paren{\widetilde{P}_{k}^{'}(0)\widetilde{P}_{k'}\widetilde{P}_{k'}^{'}(0)\widetilde{P}_{k}}+\sum_{k
\in I_j}w_{k}^{2}\text{tr}\paren{\widetilde{P}_{k}^{'}(0)S_{\mathbf{W},\mathbf{w}}\widetilde{P}_{k}^{'}(0)}\notag\\
=&\lambda_{j}\sum_{k \in
I_j}w_{k}^{2}\text{tr}\paren{\widetilde{P}_{k}^{''}(0)\widetilde{P}_{k}}+\sum_{k
\in I_j}\sum_{k' \in
  I_j}w_{k}^{2}w_{k'}^{2}\text{tr}\paren{\widetilde{P}_{k}^{'}(0)\widetilde{P}_{k'}^{'}(0)\widetilde{P}_{k'}\widetilde{P}_{k}}\notag\\
&+\sum_{k \in I_j}\sum_{k' \in
I_j}w_{k}^{2}w_{k'}^{2}\text{tr}\paren{\widetilde{P}_{k}^{'}(0)\widetilde{P}_{k'}\widetilde{P}_{k'}^{'}(0)\widetilde{P}_{k}}+\sum_{k
\in
I_j}w_{k}^{2}\text{tr}\paren{\widetilde{P}_{k}^{'}(0)S_{\mathbf{W},\mathbf{w}}\widetilde{P}_{k}^{'}(0)},
\end{align}
and
\begin{align}\label{E b(0)}
b(0)=&\sum_{k \in I_j}\sum_{k' \in I_j}w_{k}^{2}w_{k'}^{2}\text{tr}\paren{\widetilde{P}_{k}^{'}(0)\widetilde{P}_{k'}^{'}(0)\widetilde{P}_{k'}\widetilde{P}_{k}}+\sum_{k=1}^{K}\sum_{k' \in I_j}w_{k}^{2}w_{k'}^{2}\text{tr}\paren{\widetilde{P}_{k}\widetilde{P}_{k'}^{''}(0)\widetilde{P}_{k'}}+\notag\\
& +\sum_{k=1}^{K}\sum_{k' \in
I_j}w_{k}^{2}w_{k'}^{2}\text{tr}\paren{\widetilde{P}_{k}\widetilde{P}_{k'}^{'}(0)\widetilde{P}_{k'}^{'}(0)}+\sum_{k
\in I_j}\sum_{k' \in I_j}w_{k}^{2}w_{k'}^{2}\text{tr}\paren{\widetilde{P}_{k}\widetilde{P}_{k'}^{'}(0)\widetilde{P}_{k'}\widetilde{P}_{k}^{'}(0)}\notag\\
=& \sum_{k \in I_j}\sum_{k' \in I_j}w_{k}^{2}w_{k'}^{2}\text{tr}\paren{\widetilde{P}_{k}^{'}(0)\widetilde{P}_{k'}^{'}(0)\widetilde{P}_{k'}\widetilde{P}_{k}}+\sum_{k' \in I_j}w_{k'}^{2}\text{tr}\paren{S_{\mathbf{W},\mathbf{w}}\widetilde{P}_{k'}^{''}(0)\widetilde{P}_{k'}}+\notag\\
& +\sum_{k' \in
I_j}w_{k'}^{2}\text{tr}\paren{S_{\mathbf{W},\mathbf{w}}\widetilde{P}_{k'}^{'}(0)\widetilde{P}_{k'}^{'}(0)}+\sum_{k
\in I_j}\sum_{k' \in
  I_j}w_{k}^{2}w_{k'}^{2}\text{tr}\paren{\widetilde{P}_{k}\widetilde{P}_{k'}^{'}(0)\widetilde{P}_{k'}\widetilde{P}_{k}^{'}(0)}\notag\\
=&\lambda_{j}\sum_{k \in
I_j}w_{k}^{2}\text{tr}\paren{\widetilde{P}_{k}^{''}(0)\widetilde{P}_{k}}+\sum_{k
\in I_j}\sum_{k' \in
  I_j}w_{k}^{2}w_{k'}^{2}\text{tr}\paren{\widetilde{P}_{k}^{'}(0)\widetilde{P}_{k'}^{'}(0)\widetilde{P}_{k'}\widetilde{P}_{k}}\notag\\
&+\sum_{k \in
I_j}w_{k}^{2}\text{tr}\paren{\widetilde{P}_{k}^{'}(0)S_{\mathbf{W},\mathbf{w}}\widetilde{P}_{k}^{'}(0)}+\sum_{k
\in I_j}\sum_{k' \in
I_j}w_{k}^{2}w_{k'}^{2}\text{tr}\paren{\widetilde{P}_{k}\widetilde{P}_{k'}^{'}(0)\widetilde{P}_{k'}\widetilde{P}_{k}^{'}(0)}.
\end{align}
Using (\ref{E Fk* Fk = I}) and (\ref{E Hk* Fk' = 0}) along with
Lemma~\ref{L D R(A)}(2)-(3), that $\text{tr}\paren{D_{k}}=0$ and
(\ref{E sum zkl fkl 0}), we can rewrite each sum in (\ref{E a(0)})
and (\ref{E b(0)}) obtaining
\begin{align}\label{E a(0) b(0) 1}
\lambda_{j}\sum_{k \in
I_j}w_{k}^{2}\text{tr}\paren{\widetilde{P}_{k}^{''}(0)\widetilde{P}_{k}}=&\lambda_{j}\sum_{k \in I_j}w_{k}^{2}\text{tr}\paren{\corch{\widetilde{F}_{k}F_k^{*}+2H_{k}H_k^{*}+F_{k}D_{k}F_k^*+F_k\widetilde{F}_{k}^{*}}F_{k}F_{k}^{*}}\notag\\
=& \lambda_{j}\sum_{k \in I_j}w_{k}^{2}\text{tr}\paren{\widetilde{F}_{k}F_{k}^{*}+D_{k}+\widetilde{F}_{k}^{*}F_{k}}\notag\\
=& 2\lambda_{j}\sum_{k \in
I_j}w_{k}^{2}\text{tr}\paren{\widetilde{F}_{k}F_{k}^{*}} =
-2\lambda_{j}\sum_{k \in I_j}w_{k}^{2}\sum_{l=1}^{L_k}|z_{k,l}|^2,
\end{align}
\begin{align}\label{E a(0) b(0) 2}
\sum_{k \in I_j}\sum_{k' \in
I_j}w_{k}^{2}w_{k'}^{2}&\text{tr}\paren{\widetilde{P}_{k}^{'}(0)\widetilde{P}_{k'}^{'}(0)\widetilde{P}_{k'}\widetilde{P}_{k}}=\notag\\
&= \sum_{k \in I_j}\sum_{k' \in I_j}w_{k}^{2}w_{k'}^{2}\text{tr}\paren{(H_{k}F_{k}^{*}+F_{k}H_{k}^{*})(H_{k'}F_{k'}^{*}+F_{k'}H_{k'}^{*})F_{k'}F_{k'}^{*}F_{k}F_{k}^{*}}\notag\\
&=\sum_{k \in I_j}\sum_{k' \in I_j}w_{k}^{2}w_{k'}^{2}\text{tr}\paren{F_{k}F_{k}^{*}(H_{k}F_{k}^{*}+F_{k}H_{k}^{*})(H_{k'}F_{k'}^{*}+F_{k'}H_{k'}^{*})F_{k'}F_{k'}^{*}}\notag\\
&=\sum_{k \in I_j}\sum_{k' \in
I_j}w_{k}^{2}w_{k'}^{2}\text{tr}\paren{F_{k}H_{k}^{*}H_{k'}F_{k'}^{*}}=\|\sum_{k
\in
I_j}w_{k}^{2}\sum_{l=1}^{L_{k}}\overline{z_{k,l}}f_{k,l}\|^{2}=0,
\end{align}
\begin{align}\label{E a(0) 3}
\sum_{k \in I_j}\sum_{k' \in
I_j}w_{k}^{2}w_{k'}^{2}&\text{tr}\paren{\widetilde{P}_{k}^{'}(0)\widetilde{P}_{k'}\widetilde{P}_{k'}^{'}(0)\widetilde{P}_{k}}=\notag\\
&=\sum_{k \in I_j}\sum_{k' \in  I_j}w_{k}^{2}w_{k'}^{2}\text{tr}\paren{\widetilde{P}_{k}\widetilde{P}_{k}^{'}(0)\widetilde{P}_{k'}\widetilde{P}_{k'}^{'}(0)}\notag\\
&=\sum_{k \in I_j}\sum_{k' \in
I_j}w_{k}^{2}w_{k'}^{2}\text{tr}\paren{F_{k}F_{k}^{*}(H_{k}F_{k}^{*}+F_{k}H_{k}^{*})F_{k'}F_{k'}^{*}(H_{k'}F_{k'}^{*}+F_{k'}H_{k'}^{*})}\notag\\
&=\sum_{k \in I_j}\sum_{k' \in
I_j}w_{k}^{2}w_{k'}^{2}\text{tr}\paren{F_{k}H_{k}^{*}F_{k'}H_{k'}^{*}}=0,
\end{align}
\begin{align}\label{E a(0) 4 b(0) 3}
\sum_{k \in
  I_j}w_{k}^{2}\text{tr}\paren{\widetilde{P}_{k}^{'}(0)S_{\mathbf{W},\mathbf{w}}\widetilde{P}_{k}^{'}(0)}
  &= \sum_{k \in
  I_j}w_{k}^{2}\text{tr}\paren{(H_{k}F_{k}^{*}+F_{k}H_{k}^{*})S_{\mathbf{W},\mathbf{w}}(H_{k}F_{k}^{*}+F_{k}H_{k}^{*})}\notag\\
  &= \sum_{k \in
  I_j}w_{k}^{2}\text{tr}\paren{(H_{k}F_{k}^{*}+F_{k}H_{k}^{*})(\lambda_{J}H_{k}F_{k}^{*}+\lambda_{j}F_{k}H_{k}^{*})}\notag\\
  &= \sum_{k \in
  I_j}w_{k}^{2}(\lambda_{j}\text{tr}\paren{H_{k}H_{k}^{*}}+\lambda_{J}\text{tr}\paren{H_{k}^{*}H_{k}})\notag\\
  &= (\lambda_{j}+\lambda_{J})\sum_{k \in
  I_j}w_{k}^{2}\text{tr}\paren{H_{k}^{*}H_{k}}= (\lambda_{j}+\lambda_{J})\sum_{k \in
  I_j}w_{k}^{2}\sum_{l=1}^{L_{k}}|z_{k,l}|^{2},
\end{align}
and
\begin{align}\label{E b(0) 4}
\sum_{k \in I_j}\sum_{k' \in
I_j}w_{k}^{2}w_{k'}^{2}&\text{tr}\paren{\widetilde{P}_{k}\widetilde{P}_{k'}^{'}(0)\widetilde{P}_{k'}\widetilde{P}_{k}^{'}(0)}=\notag\\
&=\sum_{k \in I_j}\sum_{k' \in
I_j}w_{k}^{2}w_{k'}^{2}\text{tr}\paren{F_{k}F_{k}^{*}(H_{k'}F_{k'}^{*}+F_{k'}H_{k'}^{*})F_{k'}F_{k'}^{*}(H_{k}F_{k}^{*}+F_{k}H_{k}^{*})}\notag\\
&=\sum_{k \in I_j}\sum_{k' \in
I_j}w_{k}^{2}w_{k'}^{2}\text{tr}\paren{(F_{k}F_{k}^{*}F_{k'}H_{k'}^{*})(F_{k'}F_{k'}^{*}F_{k}H_{k}^{*})}=0.
\end{align}
By (\ref{E a(0) b(0) 1}), (\ref{E a(0) b(0) 2}), (\ref{E a(0) 3})
and (\ref{E a(0) 4 b(0) 3}) in (\ref{E a(0)}), and (\ref{E a(0) b(0)
1}), (\ref{E a(0) b(0) 2}), (\ref{E a(0) 4 b(0) 3}) and (\ref{E b(0)
4}) in (\ref{E b(0)}),
\begin{eqnarray*}
  a(0)=b(0) &=& -2\lambda_{j}\sum_{k \in
I_j}w_{k}^{2}\sum_{l=1}^{L_k}|z_{k,l}|^2+(\lambda_{j}+\lambda_{J})\sum_{k
\in
  I_j}w_{k}^{2}\sum_{l=1}^{L_{k}}|z_{k,l}|^{2} \\
  &=& (\lambda_{J}-\lambda_{j})\sum_{k \in
  I_j}w_{k}^{2}\sum_{l=1}^{L_{k}}|z_{k,l}|^{2}<0,
\end{eqnarray*}
and consequently
$\frac{d^2}{dt^2}FFP_{\mathbf{w}}(\{V_k(t)\}_{k=1}^K)|_{t=0}=
2[a(0)+b(0)] < 0$.
\end{proof}
\begin{thm}\label{T particion minimizadores FFP en OFB}
If $\mathbf{W}$ is a local minimizer of $FFP_{\mathbf{w}}$  such
that $(\mathbf{W},\mathbf{w})\in \mathcal{E}$, then for all $j<J$,
$\{(W_k,\lambda_{j}^{-1/2}w_{k})\}_{k \in I_j}$ is an orthonormal
fusion basis for $E_{j}$, in particular $w_{k}^2=\lambda_{j}$ for $k
\in I_j$.
\end{thm}
\begin{proof}
By Theorem~\ref{T particion minimizadores FFP en TFF}(5) and
 Theorem~\ref{T particion minimizadores FFP en sumas
directas}, if $\mathbf{W}$ is a local minimizer of
$FFP_{\mathbf{w}}$, then for all $j<J$, $\{W_k\}_{k \in I_j}$ is a
Riesz fusion basis for $E_{j}$ and

\centerline {$T_{j}^{*}T_{j}=\lambda_{j}I_{\bigoplus_{k \in
I_j}W_k}$,}

\noindent where $T_{j}=T_{\{W_k\}_{k \in I_j},\{w_k\}_{k \in I_j}}$.
Let $\{e_{k,l}\}_{l \in L_{k}}$ be an orthonormal basis for $W_{k}$
for all $k =1, \ldots, K$. Then $\{\{(\delta_{k,m}e_{k,l})_{m \in
I_j}\}_{l\in L_k}\}_{k \in I_j}$ is an orthonormal basis for
$\bigoplus_{k \in I_j}W_k$. If $k, k' \in I_j$, $l \in L_k$ and $l'
\in L_{k'}$, then
\begin{align*}
\langle e_{k,l}, e_{k',l'} \rangle &= \langle
\frac{1}{w_{k}}T_{j}(\delta_{k,m}e_{k,l})_{m \in I_j},
\frac{1}{w_{k'}}T_{j}(\delta_{k',m}e_{k',l'})_{m \in I_j} \rangle\\
  &= \frac{1}{w_{k}w_{k'}}\langle (\delta_{k,m}e_{k,l})_{m \in I_j},
T_{j}^{*}T_{j}(\delta_{k',m}e_{k',l'})_{m \in I_j}
\rangle\\
  &= \frac{\lambda_{j}}{w_{k}w_{k'}}\langle (\delta_{k,m}e_{k,l})_{m \in I_j},
(\delta_{k',m}e_{k',l'})_{m \in I_j} \rangle =
\frac{\lambda_{j}}{w_{k}w_{k'}}\delta_{k,k'}\delta_{l,l'}
\end{align*}
This shows that the subspaces of the family $\{W_k\}_{k \in I_j}$
are orthogonal and
$1=\|e_{k,l}\|^{2}=\frac{\lambda_{j}}{w_{k}^{2}}$. Since also
$\{W_k\}_{k \in I_j}$ is a Riesz fusion basis for $E_{j}$, the
family $\{(W_k,\lambda_{j}^{-1/2}w_{k})\}_{k \in I_j}$ is an
orthonormal fusion basis for $E_j$.
\end{proof}

\section{A concept of irregularity and the structure of the
minimizers}

The next lemma leads to a notion of irregularity of a positive
decreasing sequence, which is presented in
\cite{Massey-Ruiz-Stojanoff (2010)} and extends the one in
\cite{Casazza-Fickus-Kovacevic-Leon-Tremain (2006)}.

\begin{lem}\label{L irregularidad}
Let  $d\leq\sum_{k=1}^{K}L_{k}$. For any real positive decreasing
sequence $\{c_k\}_{k=1}^K,$ there exists a unique index $N_0$ with
$1\le N_0\le K$ such that the inequality
\begin{equation*}
(d-\sum_{k=1}^{j}L_{k})c_j >  \sum_{k=j+1}^K L_{k} c_k
\end{equation*}
holds for $1\le j <  N_0,$ while
\begin{equation}\label{irr}
(d-\sum_{k=1}^{j}L_{k})c_j \leq  \sum_{k=j+1}^K L_{k} c_k
\end{equation}
holds for $N_0\le j \le K.$
\end{lem}
\begin{proof}
We will assume that any summation over an empty set is zero. Let
$\mathcal{I}$ the set of indices such that (\ref{irr}) holds, which
is non-empty since $K \in \mathcal{I}.$ Let $j\in \mathcal{I}.$ If
$\sum_{k=1}^{j}L_{k}>d$ then it is immediate that $j+1\in
\mathcal{I}.$ If $\sum_{k=1}^{j}L_{k}\leq d$ then
\begin{align*}
[d-\sum_{k=1}^{j+1}L_{k}]c_{j+1}&=-L_{j+1}c_{j+1}+(d-\sum_{k=1}^{j}L_{k})c_{j+1}\leq
-L_{j+1}c_{j+1}+(d-\sum_{k=1}^{j}L_{k})c_{j}\\&\leq
-L_{j+1}c_{j+1}+\sum_{k=j+1}^K L_{k} c_k=\sum_{k=j+2}^K L_{k} c_k
\end{align*}
Therefore $k+1\in \mathcal{I}$. Set $N_0$ as the minimum index of
$\mathcal{I}.$
\end{proof}

We say that the \textit{$(\mathbf{L},d)$-irregularity} of a real
nonnegative decreasing sequence $\{c_k\}_{k=1}^K$ is $N_0-1,$ where
$N_0$ is the unique index from the previous Lemma.

\subsection{The fundamental inequality}

\begin{defn}\label{D desigualdad fundamental}
We say that $(\mathbf{w} ,\mathbf{L})$ satisfies the fundamental
inequality if
\begin{equation}\label{E desigualdad fundamental}
\max_{k=1,\ldots,K}w_{k}^{2} \leq \frac{1}{d}
\sum_{k=1}^{K}w_k^{2}L_{k}.
\end{equation}
\end{defn}

\begin{rem}
If $(\mathbf{w}, \mathbf{L})$ satisfies the fundamental inequality,
then
$$\max_{k=1,\ldots,K}w_{k}^{2} \leq \frac{1}{d} \sum_{k=1}^{K}w_k^{2}L_{k} \leq \frac{1}{d}
\max_{k=1,\ldots,K}w_{k}^{2}\sum_{k=1}^{K}L_{k},$$ thus
$\sum_{k=1}^{K}L_{k} \geq d$. Note that this restriction is
obviously required for fusion frames to exist for that $\mathbf{L}$.
\end{rem}

\begin{rem}\label{R desigualdad fundamental w iguales}
Let $w \in \mathbb{R}_{>0}$. Then $(w, \mathbf{L})$ satisfies the
fundamental inequality if and only if $d \leq \sum_{k=1}^{K}L_{k}$.
In particular, if $L \in \mathbb{N}$, then $(w, L)$ satisfies the
fundamental inequality if and only if $d \leq KL$.
\end{rem}

\begin{prop}\label{P fundineq}
If $(\mathbf{W},\mathbf{w})$ is an $\alpha-$tight fusion frame for
$\mathbb{F}^{d}$ with $\mathbf{W} \in \mathcal{S}_{K}(\mathbf{L})$
then $(\mathbf{w}, \mathbf{L})$ satisfies the fundamental
inequality.
\end{prop}
\begin{proof}
For $k\in\{1,\ldots,K\},$
\begin{align*}
w_k^2L_k&=w_k^2\text{tr}(\pi_{W_k})\leq \sum_{k'=1}^K w_{k'}^2
\text{tr}(\pi_{W_{k'}}\pi_{W_{k}})= \text{tr}(\sum_{k'=1}^K
w_{k'}^2\pi_{W_{k'}}\pi_{W_{k}})\\ &=
\text{tr}(S_{\mathbf{W},\mathbf{w}}\pi_{W_{k}}) = \text{tr}(\alpha
I_{\mathcal{H}}\pi_{W_{k}}) = \alpha\text{tr}(\pi_{W_{k}})= \alpha
L_{k}
\end{align*}
Hence, $w_k^2\leq \alpha$ for all $k=1,...,K.$ On the other hand,
$$\sum_{k=1}^K w_k^2 L_k =\sum_{k=1}^K w_{k}^{2}\text{tr}(\pi_{W_k})=
\text{tr}(S_{\mathbf{W},\mathbf{w}}) = \text{tr}(\alpha
I_{\mathcal{H}}) = \alpha d.$$ So the result follows.
\end{proof}

The following lemma establishes the relation between the concept of
irregularity and the fundamental inequality.

\begin{lem}\label{L relacion des fund e irregularidad}
Suppose that $\mathbf{w}$ is arranged in decreasing order. Then
$(\mathbf{w},\mathbf{L})$ satisfies the fundamental inequality if
and only if the $(\mathbf{L},d)-$irregularity of $\mathbf{w}$ is
zero.
\end{lem}

\begin{proof}
The pair $(\mathbf{w}, \mathbf{L})$ satisfies the fundamental
inequality if and only if $w_{1}^{2}\leq \frac{1}{d}\sum_{k=1}^K
L_{k}w_{k}^{2}$, i.e., $(d-L_{1})w_{1}^{2}\leq
\sum_{k=2}^KL_{k}w_{k}^{2}$, and this is equivalent to say that the
$(\mathbf{L},d)-$irregularity of $\mathbf{w}^2$ is zero.
\end{proof}

\subsection{The concept of irregularity and the structure of the
minimizers}

To prove Theorem \ref{T estructura minimos FFP con N0} we need the
following proposition.
\begin{prop}\label{T minimo no ajustado FFP IJ={1,...,K}}
Assume that $\mathbf{w}$ is arranged in decreasing order and that
$\mathbf{W}$ is a local minimizer of $FFP_{\mathbf{w}},$ such that
$(\mathbf{W},\mathbf{w})\in \mathcal{E}$. If $\mathbf{w}$ has
$(\mathbf{L},d)$-irregularity $N_{0}-1$, then $I_{J} = \{N_{0},
\ldots, K\}$.
\end{prop}

\begin{proof}
 Reasoning as in \cite{Casazza-Fickus-Kovacevic-Leon-Tremain
(2006)}, we will first show that $\{N_{0},\ldots,K\}\subseteq I_J,$
which is equivalent to prove that $I_J^{c}\subseteq
\{1,\ldots,N_{0}-1\}.$  Let $k\in I_J^{c}.$ Using Theorem~\ref{T
particion minimizadores FFP en TFF}(1)(3)(4) and Theorem~\ref{T
particion minimizadores FFP en OFB},
\begin{equation*}
  \begin{aligned}
 \sum_{i=k+1}^{K}L_{i}w_i^2&= w_k^2\sum_{i=1, i \in I_{J}}^{k}L_{i} + \sum_{i=k+1}^{K}L_{i}w_i^2-w_k^2\sum_{i=1, i \in I_{J}}^{k}L_{i}\\
  &\leq \sum_{i=1, i \in I_{J}}^{k}L_{i}w_i^2 + \sum_{i=k+1}^{K}L_{i}w_i^2-w_k^2\sum_{i=1, i \in I_{J}}^{k}L_{i}\\
  &\leq \sum_{i \in I_J}L_iw_i^2 + \sum_{i=k+1, i\in I_J^{c}}^{K} L_{i}w_i^2 -w_k^2\sum_{i=1, i \in I_{J}}^{k}L_{i}\\
  &=\lambda_J \text{dim}(E_{J}) + \sum_{i=k+1, i\in I_J^{c}}^{K} L_{i}w_i^2-w_k^2\sum_{i=1, i \in I_{J}}^{k}L_{i}\\
  &=\lambda_J(d-\sum_{j=1}^{J-1}\sum_{k \in I_{j}}L_{k})+\sum_{i=k+1, i\in I_J^{c}}^{K} L_{i}w_i^2-w_k^2\sum_{i=1, i \in I_{J}}^{k}L_{i}\\
  &<(d-\sum_{k \in I_{J}^{c}}L_{k}+\sum_{i=k+1, i\in I_J^{c}}^{K} L_{i}-\sum_{i=1, i \in I_{J}}^{k}L_{i})w_k^2\\
  &=(d-\sum_{i=1}^{k}L_{i}) w_k^2.\\
\end{aligned}
\end{equation*}
Hence $k\in \{1,\ldots,N_{0}-1\}.$

To see that $I_J\subseteq \{N_{0},\ldots,K\}$ it is only left to prove that $\{1,\ldots,N_{0}-1\}\cap I_J=\emptyset.$\\
Assume to the contrary that $\{1,\ldots,N_{0}-1\}\cap
I_J\neq\emptyset.$ Let $k_0=\min \{1,\ldots,N_{0}-1\}\cap I_J$ and
$k_1=\max \{1,\ldots,N_{0}-1\}\cap I_J$. We have $k_0=\min I_J$, so
$\{k_{0}\}=\{1,\ldots,k_{0}\}\cap I_{J}$. Thus, by Theorem~\ref{T
particion minimizadores FFP en TFF}(5) and Proposition~\ref{P
fundineq},
\begin{equation*}
  \begin{aligned}
w_{k_{0}}^2 = \max_{k \in I_{J}}w_{k}^2 &\le
\frac{1}{\text{dim}(E_{J})}
\sum_{k\in I_J}L_{k} w_k^2\\
(\text{dim}(E_{J})-L_{k_{0}})w_{k_{0}}^2&\le \sum_{k\in I_J,k > k_{0}}L_{k}w_k^2\\
(\text{dim}(E_{J})-\sum_{k\in \{1,\ldots,k_{0}\}\cap I_{J}}L_{k})w_{k_{0}}^2&\le \sum_{k\in I_J,k > k_{0}}L_{k}w_k^2.\\
\end{aligned}
\end{equation*}
Hence, by Theorem~\ref{T particion minimizadores FFP en TFF}(1)(3)
and Theorem~\ref{T particion minimizadores FFP en OFB},
\begin{align*}
(d-\sum_{k\in I_{J}^{c}}L_{k}-\sum_{k\in \{1,\ldots,k_{1}\}\cap
I_{J}}L_{k})w_{k_{1}}^2&=(d-\sum_{j=1}^{J-1}\sum_{k \in
I_{j}}L_{k}-\sum_{k\in \{1,\ldots,k_{1}\}\cap
I_{J}}L_{k})w_{k_{1}}^2\\ &= (\text{dim}(E_{J})-\sum_{k\in
\{1,\ldots,k_{1}\}\cap I_{J}}L_{k})w_{k_{1}}^2\\ & \leq \sum_{k\in
I_J,k > k_{1}}L_{k}w_k^2.
\end{align*}
Since $\{N_{0},\ldots,K\} \subseteq I_{J}$ and $k_1$ is the maximum
of $\{1,\ldots,N_{0}-1\}\cap I_J$, we have
$\{1,\ldots,N_{0}-1\}=I_{J}^{c} \cup (\{1, \ldots, k_{1}\} \cap
  I_J)$, and thus, $d-\sum_{k=1}^{N_{0}-1}L_{k} = d-\sum_{k\in
I_{J}^{c}}L_{k}-\sum_{k\in \{1,\ldots,k_{1}\}\cap I_{J}}L_{k}.$
Also, $\{k_1+1,\ldots,K\}\cap
  I_J=\{N_{0},\ldots,K\}$.

Therefore, from the previous inequality,
$(d-\sum_{k=1}^{N_{0}-1}L_{k})w_{N_{0}-1}^2\le
\sum_{k=N_{0}}^{K}L_{k}w_k^2.$ But the definition of the
irregularity $N_{0}$  gives that $(d-\sum_{i=1}^{k}L_{i})w_k^2 >
\sum_{i=k+1}^KL_{i}w_i^2$ for all $k< N_{0},$ which is a
contradiction for $k=N_{0}-1.$
\end{proof}

If $\mathbf{W}$ is a local minimizer of $FFP_{\mathbf{w}}$such that
$(\mathbf{W},\mathbf{w})\in \mathcal{E},$ then Proposition~\ref{T
minimo no ajustado FFP IJ={1,...,K}} shows that $I_{J} \neq
\emptyset$. So we can state the following:

\begin{cor}\label{C minimo local entonces genera}
If $\mathbf{W}$ is a local minimizer of $FFP_{\mathbf{w}}$such that
$(\mathbf{W},\mathbf{w})\in \mathcal{E},$ then
$(\mathbf{W},\mathbf{w})$ is a fusion frame.
\end{cor}

\begin{thm}\label{T estructura minimos FFP con N0}
Let $\mathbf{L}\in \N^K$ and $\mathbf{w}$ be a vector of decreasing
weights with $(\mathbf{L},d)$-irregularity equal to $N_{0}$. Then
any local minimizer $\mathbf{W}$ of the frame potential
$FFP_{\mathbf{w}}$ such that $(\mathbf{W},\mathbf{w})\in
\mathcal{E},$ can be decomposed as
$$(\mathbf{W},\mathbf{w})=\{(W_k,w_{k})\}_{k=1}^{N_0-1}\cup
\{(W_k,w_{k})\}_{k=N_0}^{K},$$ where
$\{(W_k,\lambda_{j}^{-1/2}w_{k})\}_{k=1}^{N_0-1}$ is an orthonormal
fusion frame sequence for whose orthogonal complement the sequence
$\{(W_k,w_{k})\}_{k=N_0}^{K}$ is a tight fusion frame.
\end{thm}

\begin{proof}It follows from Theorem~\ref{T particion minimizadores FFP en TFF}, Theorem~\ref{T particion minimizadores FFP en OFB}, Proposition~\ref{T minimo no ajustado FFP
IJ={1,...,K}} and Corollary~\ref{C minimo local entonces genera}.
\end{proof}

\begin{cor}\label{C valor minimo FFP con N0}
Let $\mathbf{L}\in \N^K$ and $\mathbf{w}$ be a vector of decreasing
weights with $(\mathbf{L},d)$-irregularity equal to $N_{0}$. Then
any local minimizer of the fusion frame potential
$FFP_{\mathbf{w}},$ which belongs to $ \mathcal{E}$, is also a
global minimizer and the minimum value is
\begin{equation}\label{minval}
\sum_{k=1}^{N_0-1}w_k^4L_k+\frac{1}{d-\sum_{k=1}^{N_0-1}L_k}\left(\sum_{k=N_0}^{K}w_k^2L_k\right)^2.
\end{equation}
\end{cor}

\begin{proof}
Let $\mathbf{W}$ be any local minimizer in $ \mathcal{E}.$ Then we
have the decomposition of Theorem~\ref{T estructura minimos FFP con
N0}. On the one hand, $\pi_{W_k}\pi_{W_{k'}}=0$ for $k'\neq
k=1,\ldots,N_0-1,$ and hence
$$\sum_{k,k^{\prime}=1}^{N_{0}-1}
w_k^2w_{k\prime}^2\text{tr}(\pi_{W_{k}}\pi_{W_{k^{\prime}}})=\sum_{k=1}^{N_0-1}
w_k^4L_k.$$ On the other hand,
$\text{dim}(\text{span}\bigcup_{k=1}^{N_0-1}W_k)=\sum_{k=1}^{N_0-1}L_k$,
so
$\text{dim}(\text{span}\bigcup_{k=N_0}^{K}W_k)=d-\sum_{k=1}^{N_0-1}L_k$
and by Proposition~\ref{cotainf},
$$\sum_{k,k^{\prime}=N_{0}}^{K}
w_k^2w_{k\prime}^2\text{tr}(\pi_{W_{k}}\pi_{W_{k^{\prime}}})=\frac{1}{d-\sum_{k=1}^{N_0-1}L_k}\left(\sum_{k=N_0}^{K}w_k^2L_k\right)^2.$$
Therefore
$$FFP_{\mathbf{w}}(\mathbf{W},\mathbf{w})=\sum_{k=1}^{N_0-1}
w_k^4L_k
+\frac{1}{d-\sum_{k=1}^{N_0-1}L_k}\left(\sum_{k=N_0}^{K}w_k^2L_k\right)^2.$$
We have proved that every local minimizer in $ \mathcal{E}$  of the
fusion frame potential attains the same value (\ref{minval}). By
Proposition~4.1.2 of \cite{Massey-Ruiz-Stojanoff (2010)},
(\ref{minval}) is a lower bound for $FFP_{\mathbf{w}}.$ Hence every
local minimizer is also a global minimizer and it attains the value
given in (\ref{minval})
\end{proof}

\begin{rem} It can be seen that if a fusion frame $(\mathbf{W},\mathbf{w})$  has
the structure described in Theorem~ \ref{T estructura minimos FFP
con N0}, then $(\mathbf{W},\mathbf{w})\in \mathcal{E}$. In
\cite[Proposition~4.1.2]{Massey-Ruiz-Stojanoff (2010)} it is shown
that $(\mathbf{W},\mathbf{w})$ has the structure described in
Theorem~ \ref{T estructura minimos FFP con N0} if and only if the
value given in (\ref{minval}) is attained.
\end{rem}

\section{Existence of tight fusion frames}

Tight fusion frames, as well as tight frames, are particular useful
mainly because they allow a very simple and computational efficient
representation of the elements of $\mathcal{H}$. In this section we
relate the the existence of tight fusion frames with the minimizers
of the fusion frame potential and the fundamental inequality.

\begin{prop}\label{P (w,L) des fund entonces minimo FFP ajustado}
Let $(\mathbf{w}, \mathbf{L})$ satisfy the fundamental inequality
and let $\mathbf{W}$ be a local minimizer of the fusion frame
potential $FFP_{\mathbf{w}},$ such that $(\mathbf{W},\mathbf{w})\in
\mathcal{E}$. Then $(\mathbf{W},\mathbf{w})$  is a tight fusion
frame.
\end{prop}
\begin{proof}
Let $(\mathbf{w}, \mathbf{L})$ satisfy the fundamental inequality
and $\mathbf{W}$ be a local minimizer of $FFP_{\mathbf{w}}.$ By
Theorem~\ref{T particion minimizadores FFP en TFF} and
Theorem~\ref{T particion minimizadores FFP en OFB}, if
$(\mathbf{W},\mathbf{w})$ is not tight, then $J>1$ and
\begin{equation*}
\begin{aligned}
\sum_{k=1}^Kw_k^2L_k&=\sum_{j=1}^{J-1}\sum_{k\in I_j}w_k^2L_k+\sum_{k\in I_J}w_k^2L_k =\sum_{j=1}^{J-1}\sum_{k\in I_j}\lambda_jL_k+\lambda_J\text{dim}(E_{J})\\
&< \lambda_1\sum_{j=1}^{J-1}\sum_{k\in I_j}L_k+\lambda_1 \text{dim}(E_{J})=\lambda_1(\sum_{j=1}^{J-1}\text{dim}(E_{j})+\text{dim}(E_{J}))\\
&=\lambda_1 (d-\text{dim}(E_{J})+\text{dim}(E_{J})) =w_1^2d,
\end{aligned}
\end{equation*}
which contradicts that $(\mathbf{w}, \mathbf{L})$ satisfies the
fundamental inequality. Therefore,  $(\mathbf{W},\mathbf{w})$  is a
tight fusion frame.
\end{proof}

We will consider the following class of pairs of weights and
dimensions:
\begin{align}
\mathcal{M}=\{(\mathbf{w}, \mathbf{L}):\text{there exists a
local minimizer of the } &\text{frame potential }\notag\\
&FFP_{\mathbf{w}} \text{ that belongs to } \mathcal{E}\}.\notag
\end{align}
As a consequence of Proposition~\ref{P fundineq} and
Proposition~\ref{P (w,L) des fund entonces minimo FFP ajustado} we
obtain:

\begin{thm}\label{T existencia TFF sii fund ineq}

Assume that $(\mathbf{w}, \mathbf{L})\in \mathcal{M}$ There exists a
tight fusion frame $(\mathbf{W},\mathbf{w})$ for $\mathbb{F}^{d}$
with $\mathbf{W} \in \mathcal{S}_{K}(\mathbf{L})$ if and only if
$(\mathbf{w}, \mathbf{L})$ satisfies the fundamental inequality.
\end{thm}

In \cite{Massey-Ruiz-Stojanoff (2010)} a characterization is given
under the restriction that $\sum_{k=1}^K w_k^2L_k=1,$  in the form
of Horn-Klyachko compatibility conditions. It can be complicated to
use in practice, since it involves the computation of the
Littlewood-Richardson coefficients of certain associated partitions.

\begin{rem}\label{contraejemplo}
The previous theorem can be used to show that some pairs
$(\mathbf{w}, \mathbf{L})$ are not in $\mathcal{M}.$ For instance,
consider $d=3$, $K=L=2$ and $w_1=w_2=1.$ In this case $d\leq KL.$ We
are going to see that $(1,2)\notin \mathcal{M}$. Assume by
contradiction that $(1,2)\in \mathcal{M}$. By Theorem~\ref{T
existencia TFF sii fund ineq}, there exists a tight fusion frame
with fusion frame bounds $\frac{KL}{d}=\frac{4}{3}.$ Let $P_1$,
$P_2$ be orthogonal projections such that
\begin{equation}\label{suma de proy.}
\frac{4}{3}I=P_1+P_2.
\end{equation}

Let $U_1$ be a unitary matrix and let $D_1$ a diagonal matrix such
that $P_1=U_1D_1U_1^{\ast}$.  From (\ref{suma de proy.}),
$\frac{4}{3}I=D_1+U_1^{\ast}P_2U_1$ where $D_1$ and
$U_1^{\ast}P_2U_1$ are orthogonal projections. The elements of the
diagonal satisfy that $\alpha^2=\alpha,$ therefore $\alpha=0$ or
$\alpha=1$. Since, $\text{tr}(D_1)=\text{tr}(P_1)=2$ we can assume
$D_1=\left(\begin{array}{ccc}
                    1 & 0 & 0 \\
                    0 & 1 & 0 \\
                    0 & 0 & 0 \\
                  \end{array}
                \right)$. Thus, $U_1^{\ast}P_2U_1=\left(
                           \begin{array}{ccc}
                             \frac{1}{3} & 0 & 0 \\
                             0 & \frac{1}{3} & 0 \\
                             0 & 0 & \frac{4}{3} \\
                           \end{array}
                         \right)$ which is not idempotent, and hence it is
not an orthogonal projection.
\end{rem}

\begin{rem}
A frame can be seen as a fusion frame where all the subspaces are
one-dimensional and the weights are equal to the norms of the
elements of the frame. We note that in this case (\ref{E desigualdad
fundamental}) coincides with the fundamental inequality for frames
introduced in \cite{Casazza-Fickus-Kovacevic-Leon-Tremain (2006)},
where it has been proved that $(\mathbf{w},1)$ always belongs to
$\mathcal{M}.$ Thus Theorem~\ref{T existencia TFF sii fund ineq}
shows that the fundamental inequality (\ref{E desigualdad
fundamental}) that involves the weights and the dimensions of the
subspaces of fusion frames plays the same role as the one that
involves the norms of frames.

If $\sum_{k=1}^K L_k=d$, set $\mathbf{W}$ such that $\mathbb{F}^{d}$
is the orthogonal sum of its elements. Clearly,
$(\mathbf{W},\mathbf{w})\in \mathcal{E}$ and
$FFP_{\mathbf{w}}(\mathbf{W},\mathbf{w})=\sum_{k=1}^{K}w_{k}^{4}L_{k}$.
If $w_{1}=\ldots=w_{K}$, then the $(\mathbf{L},d)$-irregularity of
$\mathbf{w}$ is $0$.  If $N_{0}$ is such that $w_{1}\geq\ldots\geq
w_{N_{0}-1}
> w_{N_{0}}=\ldots=w_{K}$, then the $(\mathbf{L},d)$-irregularity
of $\mathbf{w}$ is $N_{0}-1$. In both cases, the value of
$FFP_{\mathbf{w}}(\mathbf{W},\mathbf{w})$ coincides with
(\ref{minval}), so, by
\cite[Proposition~4.1.2]{Massey-Ruiz-Stojanoff (2010)}, $\mathbf{W}$
is a global minimizer. Therefore, if $\sum_{k=1}^K L_k=d$, the pair
$(\mathbf{w}, \mathbf{L})\in \mathcal{M}$ for any $\mathbf{w}$.
\end{rem}


\end{document}